\documentclass[12pt]{article}

\usepackage{amsmath,amssymb,amsthm,amsfonts}

\newtheorem{theorem}{Theorem}

\newtheorem{lemma}[theorem]{Lemma}

\newtheorem*{claim}{Claim}
\newtheorem{question}{Question}
\newtheorem*{example}{Example}

\begin{document}

\title{Biclique Covers and Partitions}
\author{Trevor Pinto\thanks{School of Mathematical Sciences, Queen Mary University of
London, London E1 4NS, UK (\tt{tpinto@maths.qmul.ac.uk})}}
\maketitle

\begin{abstract} 

The \emph{biclique cover number} (resp. \emph{biclique partition number}) of a graph $G$, $\mathrm{bc}(G$) (resp. $\mathrm{bp}(G)$), is the least number of bicliques---complete bipartite subgraphs---that are needed to cover (resp. partition) the edges of $G$.

The \emph{local biclique cover number} (resp. \emph{local biclique partition number})  of a graph $G$, $\mathrm{lbc}(G$) (resp. $\mathrm{lbp}(G)$), is the least $r$ such that there is a cover (resp. partition) of the edges of $G$ by bicliques with no vertex in more than $r$ of these bicliques.

We show that $\mathrm{bp}(G)$ may be bounded in terms of $\mathrm{bc}(G)$, in particular, $\mathrm{bp}(G)\leq \frac{1}{2}(3^\mathrm{bc(G)}-1)$. However, the analogous result does not hold for the local measures. Indeed, in our main result, we show that $\mathrm{lbp}(G)$ can be arbitrarily large, even for graphs with $\mathrm{lbc}(G)=2$. For such graphs, $G$, we try to bound $\mathrm{lbp}(G)$ in terms of additional information about biclique covers of $G$. We both answer and leave open questions related to this.

There is a well known link between biclique covers and subcube intersection graphs. We consider the problem of finding the least $r(n)$ for which every graph on $n$ vertices can be represented as a subcube intersection graph in which every subcube has dimension $r$. We reduce this problem to the much studied question of finding the least $d(n)$ such that every graph on $n$ vertices is the intersection graph of subcubes of a $d$-dimensional cube.

\end{abstract}

\section{Introduction} 

\subsection{Biclique covers and partitions}

A \emph{biclique cover} of a (simple) graph, $G$, is a collection of bicliques (complete bipartite subgraphs) of $G$, the union of whose edges is $E(G)$. A related notion is that of a \emph{biclique partition}: a collection of bicliques of $G$, whose edges partition the edges of $G$.  We say that a biclique  cover (resp. partition) is a $k$\emph{-cover} (resp. $k$\emph{-partition}) if it contains at most $k$ bicliques.   The \emph{biclique cover number} of $G$, $\mathrm{bc}(G)$, is the least $k$ for which there exists a $k$-cover of $G$. Analogously, the \emph{biclique partition number} of $G$, $\mathrm{bp}(G)$, is the least $k$ for which there exists a $k$-partition of $G$. The biclique cover number and biclique partition number have been studied extensively, see for instance \cite{mubayi} and \cite{juknakulikov}, motivated by applications to many other areas such as combinatorial geometry \cite{alon}, communication complexity \cite{comcom}, network addressing \cite{grahampollak} and even immunology \cite{nau}. %(GIVE AN ORIGINATING REFERENCE?)

%Trivially, $\mathrm{bc}(G) \leq \mathrm{bp}(G)$. It is known that this inequality can be very loose in certain cases: the Graham-Pollak Theorem, which concerns K_n, the complete graph on n vertices,  shows that $\mathrm{bp}(K_n)=n-1$, while it is easy to show that $\mathrm{bc}(K_n)=\lceil \log_2 n \rceil$.
%Indeed, enumerate the bicliques in the cover, $B_1,\dots B_k$, say, and note we can assume each is maximal, so every vertex is in the $0^{th}$ class of the $i^{th}$ biclique or in the $1^{st}$ class. Represent each vertex as a binary string of the classes of bicliques it is in. Then as every pair of vertices must differ in (at least) one digit of the representation, $bc(G_n) \geq \lceil \log_2 (n) \rceil$. The converse is proved simply by reversing this process. %(IS THIS CLEAR?).

One of the early results on the biclique partition number is the Graham-Pollak Theorem \cite{grahampollak}---see \cite{tverberg} for an elegant proof---which concerns $K_n$, the complete graph on $n$ vertices. More specifically, it states that $\mathrm{bp}(K_n)=n-1$. In contrast to this, it is easy to show that $\mathrm{bc}(K_n)=\lceil \log  n \rceil$. (Here and elsewhere, $\log$  refers to $\log_2$). These results demonstrate that the trivial inequality $\mathrm{bc}(G)\leq \mathrm{bp}(G)$ may be quite loose: $\mathrm{bp}(K_n) \geq  2^{\mathrm{bc}(K_n)-1}-1$.

%I HAVE COMMENTED OUT THE PROOF OF THIS.

In Section 3, we investigate how large $\mathrm{bp}$ can be in terms of $\mathrm{bc}$. Indeed, we prove that if $\mathrm{bc}(G)=t$, $\mathrm{bp}(G) \leq (3^t-1)/2$ and exhibit a graph showing that this is tight.

We call a biclique  cover (resp. partition) \emph{$r$-local} if every vertex is in at most $r$ of the bicliques involved in the cover (resp. partition) of $G$. The \emph{local biclique cover number}, $\mathrm{lbc}(G)$, is the least $r$ such that there exists an $r$-local cover of the graph. Similarly, the \emph{local biclique partition number}, $\mathrm{lbp}(G)$, is the least $r$ such that there exists an $r$-local partition of the graph. Variants of $\mathrm{lbp}$ and $\mathrm{lbc}$ have long been studied, starting in 1967 with Katona and Szemer\'edi \cite{katonaszem} in connection with diameter 2 graphs with few edges. An easy corollary of their result is:

\[\text{lbc}(K_n)=\mathrm{lbp}(K_n)=\lceil \log n \rceil,\] 
which was also shown in full by Dong and Liu \cite{dongliu}. This result is perhaps surprising given that $\mathrm{bc}(K_n)$ and $\mathrm{bp}(K_n)$ differ so vastly.

%Given a biclique cover or partition, $\pi$, of the edges of $G$, and a vertex $x$, the \emph{valency of $x$ with respect to $\pi$} is the number of bicliques in $\pi$ which include $x$ as a vertex. This is denoted by $v(x, \pi)$. When the cover or partition used is clear, we often refer to this simply as the valency of $x$, $v(x)$. It is easy to see that both 

 %\[ \mathrm{lbc}(G)=\min \left\{  \max_{x \in V(G)} v(x, \pi) \Big| \: \kappa\text{ is a cover of }  G \right \} \] and 
%\[\mathrm{lbp}(G)=\min \left\{\max_{x \in V(G)} v(x, \pi) \Big| \: \pi \text{ is a partition of } G \right\}. \]

One of the key results in the theory of the local biclique cover/ partition numbers, is that of Lupanov \cite{lupanov} in 1956, stating that for all $n$, there exists a graph $G$ on $n$ vertices with $\mathrm{lbc}(G)\geq \frac{c_1 n}{\log n}$, where $c_1$ is a positive constant. Indeed his result is slightly stronger than stated---he actually showed the same bound holds if $\mathrm{lbc}(G)$ is replaced by $w(G)$, the minimum over all covers, $\kappa$, of $G$ of the average number of bicliques in $\kappa$ to which each vertex belongs. This result was later reproved by Chung, Erd\H{o}s and Spencer \cite{chungetal} in 1983 and Tuza \cite{tuza} in 1984, among others.  

Each of the above papers also included corresponding upper bounds for $w(G)$, but it was not until Erd\H{o}s and Pyber \cite{erdospyber} in 1997 that a similar upper bound for $\mathrm{lbp}$, or $\mathrm{lbc}$, was proved. More formally, they showed that for all $n$, $\mathrm{lbp}(G)\leq \frac{c_2 n }{\log{n}}$, for some constant $c_2$. Combined with Lupanov's result above, this means that if $\mathrm{lbc}(G)$ is large compared to the number of vertices of $G$, then $\mathrm{lbp}(G)$ is not too much larger.

This is one motivation for seeking to bound $\mathrm{lbp}$ in terms of $\mathrm{lbc}$, hoping for a similar result to our previously mentioned bound for $\mathrm{bc}$ in terms of $\mathrm{bp}$. Perhaps surprisingly, however, we find that no such bound is possible. More specifically, in Section 4 we exhibit, for all $k\geq 2$, a graph $G$ with $\mathrm{lbc}(G)=2$, but $\mathrm{lbp}(G)\geq k$.

Given this, it is natural to ask `how fast' can $\mathrm{lbp}(G)$ tend to infinity if $\mathrm{lbc}(G)=2$. One way of formalizing this question to ask `if G has a 2-local $m$-cover, how large can $\mathrm{lbp}(G)$ be?' While we are able to answer this, up to a constant factor, in Section 5, we leave open the equally natural question of `if $\mathrm{lbc}(G)=2$ and $\mathrm{bc}(G)=m$, how large can $\mathrm{lbp}$ be?' A second question that we leave open is whether, loosely speaking, all graphs have a single cover that is close to optimal for both $\mathrm{bc}$ and $\mathrm{lbc}$.

%IF PROVE $\mathrm{lbp}$-BC result, discuss all POSSIBLE INEQUALITIES BETWEEN PAIRS. SAY THAT $\mathrm{lbp}$ CAN'T BE BOUNDED BY KNOWLEDGE OF LBC, CAN IT BE BOUNDED BY KNOWLEDGE OF BC?

\subsection{Subcube Intersection Graph Representations}

The biclique cover number discussed above has a natural interpretation in the setting of hypercube intersection graphs. 

The $d$-\emph{dimensional hypercube}, $Q_d$, is the graph with vertex set $\{0,1\}^d$ and with two vertices connected by an edge if and only if the vectors representing them differ in exactly one coordinate. A subset, $S$, of the vertices of the hypercube, $Q_d$, is called a subcube if it induces a graph isomorphic to a hypercube of some dimension. In other words, there is some set  $J\subseteq [d]=\{1,2,3,...,d\}$, and constants $a_j\in \{0,1\}$ for each $j\in J$ such that $(x_1,...,x_n) \in$ $S$ if and only if for all $j\in J$, $x_j=a_j$. \emph{Fixed coordinates} are those coordinates in $J$, whereas \emph{free coordinates} are coordinates that are not fixed.

A subcube is said to have dimension $r$ if it has $r$ free coordinates. We shall write subcubes as vectors in $\{0,1,*\}^d$, where the asterisks denote free coordinates and the 0's and 1's denote fixed coordinates, $j$, with $a_j=0$ or 1 respectively.

An \emph{intersection graph} of a family of subsets of some groundset has one vertex for each of the sets, and an edge between two vertices if and only if the corresponding sets intersect.  See \cite{jukna} or \cite{mckee} for more on intersection graphs for various set families. In this paper, we are interested in intersection graphs where all sets in the family are subcubes of a hypercube. Such graphs are called \emph{subcube intersection graphs}. Note that two subcubes intersect if and only if they agree on all coordinates where both are fixed. See Johnson and Markstr\"om \cite{johnsonmarkstrom} for some background. 

Let $I(n,d)$ be the set of all graphs on $n$ vertices that are the intersection graph of some family of subcubes of a $d$-dimensional hypercube. %Let $I(n,d,r)$ be the set of all graphs on $n$ vertices that are the intersection graph of some family of $r$-dimensional subcubes of a $d$-dimensional hypercube.
It is easy to see that  $G\in I(n,d)$ if and only if its complement has a $d$-cover, as was apparently first pointed out by Fishburn and Hammer \cite{fishburn}, and also noted in \cite{johnsonmarkstrom}. Indeed, a representation of $G\in I(n,d)$ assigns a vector in $\{0,1,*\}^d$ to each vertex. Generate a biclique for each of the $d$ coordinates, having as one class all the vertices with a zero in this coordinate, and as the other class all vertices with a one in this coordinate. Then, the union of these bicliques has an edge exactly where $G$ does not have an edge. This produces a $d$-cover of $G^c$ and the converse is similar. 

If $G$ is a graph on $n$ vertices, we write $\rho(G)$ for the smallest $d$ such that $G$ $\in$ $I(n,d)$---in other words, $\rho(G)$ is the smallest $d$ such that $G$ may be represented as an intersection graph of a family of subcubes of $Q_d$. By the previous paragraph, $\rho(G)=\mathrm{bc}(G^c)$. We also write  $\rho(n)=\max\{\rho(G): |G| =n\}$, where $|G|$ denotes the number of vertices of $G$.

Similarly, we define $\tau(G)$ to be the smallest $r$ such that $G$ may be represented as an intersection graph of a family of $r$ dimensional subcubes of some hypercube. We also let $\tau(n)=\max\{\tau(G): |G|=n\}$. Using the relationship between subcube intersection graphs and biclique coverings, we may see that $\tau(G)$ is  the least $r$ such that $G^c$ has a biclique cover where every vertex lies in all but $r$ of the bicliques in the cover. (Note in this interpretation, bicliques with one empty class are allowed). This is bears a resemblance to, but is distinct from $\mathrm{lbc}(G^c)$---the least $r$ such that $G^c$ has a cover with every biclique in no more than $r$ bicliques of the cover.

Many authors have placed various bounds on $\rho$---see for instance \cite{chung}, \cite{erdos}, \cite{rodl} and \cite{tuza}.  The best known bounds are: 
 \[n-c \log n \leq \rho(n) \leq n- \lfloor \log n\rfloor+1, \]
where $c$ is some positive constant. The upper bound is due to Tuza \cite{tuza} and the probabilistically proved lower bound to R\"odl and Ruci\'nski \cite{rodl}. The question of obtaining similar bounds for $\tau(n)$ was posed by Johnson and Markstr\"om \cite{johnsonmarkstrom}. In Section 2, we shall prove a very close relationship between $\tau(n)$ and $\rho(n)$ and thus we obtain bounds on $\tau$ similar to those on $\rho$.

We remark briefly that the other biclique covering/partitioning measures may be interpreted in this way. For instance, $\mathrm{lbc}(G^c)$ is the least $r$ such that $G$ can be represented as the intersection graph of a family of codimension $r$ subcubes of some hypercube. Equally, $\mathrm{bp}(G^c)$ is the least $d$ such that $G$ can be represented in $I(n,d)$ such that the subcubes of non-adjacent vertices differ in only one coordinate. %(IS THIS CLEAR?) SHOUld I INTERPRET \mathrm{lbp}?

\section{Relationship between $\tau$ and $\rho$} \label{subcubes}

We shall prove the upper bounds and lower bounds on $\tau$ separately.

\begin{lemma}

For any graph $G$, $\tau(G) \leq \rho(G)$, and hence $\tau(n) \leq \rho(n)$.

\end{lemma}

\begin{proof}
Let $G$ be a graph on $n$ vertices with $\rho(G)=d$. This means that there are subcubes $A_1, \dots , A_n$ of $Q_d$ whose intersection graph is $G$. Let $M$ (resp. $m$) be the maximum (resp. minimum) dimension of these subcubes. For all $i$, replace $A_i$ by $A'_i$, a subcube of $Q_{d+M-m}$ by appending 0's and $*$'s to the vector of $A_i$. We ensure that we add as many $*$'s as needed so the resultant vector has precisely $M$ $*$'s and as many 0's as needed to give the vector length $(d+M-m)$. The intersection graph of the $A'_i$ is $G$. This shows that $G$ may be represented as an intersection graph of a family of subcubes of dimension M. Since $M\leq d$, the result follows.
\end{proof}

\begin{lemma}
$\tau(n+1) \geq \rho(n)$.
\end{lemma}

\begin{proof}
Let $G$ be a graph with vertices $v_1, \dots, v_n$ and with $\rho(G)=\rho(n)$. Now form $G'$ from $G$ by adding a single vertex, $v_{n+1}$, adjacent to all vertices of $G$. Let $A_1,\dots,A_n, A_{n+1}$ be $r$-dimensional subcubes of $Q_d$ such that $G'$ is the intersection graph of the $A_i$ (the vertex $v_i$ being represented by $A_i$) and such that $r$ is equal to $\tau(G')$. Without loss of generality, let $A_{n+1}$ be free in the first $r$ coordinates, and hence fixed in the other $d-r$ coordinates. Where $A_{n+1}$ is fixed, all the other $A_i$ must be either free (have an asterisk in that coordinate) or have the same fixed value as $A_{n+1}$, as $v_{n+1}$ is adjacent to all the other vertices. Thus restricting the first $n$ subcubes to the first $r$ coordinates does not change which pairs of subcubes intersect. So the intersection graph of these restricted subcubes is $G$, implying that $G \in I(n,r)$.

Therefore, using the definition of $\rho$,  $\tau(n+1)\geq \tau(G') =r \geq \rho(G)=\rho(n)$.

\end{proof}

Combining these two lemmas with the bounds on $\rho$ gives:

\begin{theorem}
 There is some absolute constant, $c$, such that for all $n$, \[n- c \log n \leq \tau(n) \leq n- \lfloor{\log n}\rfloor+1.\] 
\end{theorem}

%WHERE SHOULD WE PUT THIS?
%IMPORTATNT LEMMA MAY BE USED LATER,

%\begin{lemma}
%For all $K_{t,t}$-free  graphs $G$ on $n$ vertices, $lbc(G)\geq e(G)/n(t-1)$.
%\end{lemma}

%\begin{proof}

%Fix a covering of the edges of $G$ by bicliques. Now colour half-edges (thought of as a pair of a vertex and an edge through it) red/blue as follows: a half-edge $(v,e)$ is coloured red if one of the bipartite graphs covering the edge contains $v$ in a class with fewer than $t$ vertices. Otherwise, colour the half-edge blue. As $G$ is $K_{t,t}-$free, for each edge $e=(u,v)$ of $G$, at least one of $(u,e)$ and $(v,e)$ is red. Therefore there are at least $e(G)$ red half-edges. Therefore there is a vertex in at least $e(G)/n$ red edges. This vertex is in at least $e(G)/n(t-1)$ bicliques, as at most $t-1$ red half-edges come from each biclique.

%\end{proof}

%\textbf{Possible open question} - can one cover each graph on $n$ vertices with fewer than $n$ bipartite graphs, with each vertex in fewer than $cn/\log( n)$ for a constant $c$. Note that our method in the upper bound of the theorem requires $n^A$ different bicliques, where $A>1$, and same for the method of Erdős and Pyber.

\section{Bounding $\mathrm{bp}$ in terms of $\mathrm{bc}$} \label{bounding bp}

As the earlier example of $K_n$ shows, $\mathrm{bp}$ can be as large as exponential in $\mathrm{bc}$. Here, we show that $\mathrm{bp}$ grows no faster than exponentially in $\mathrm{bc}$, and calculate the best upper bound exactly.
We do this by reducing the problem to proving the upper bound for a single graph. In the second theorem  of this section, we use ideas from Tverberg's proof \cite{tverberg} of the Graham-Pollak Theorem to calculate $\mathrm{bp}$ of this graph and show that the previous bound is tight.

\begin{theorem} \label{bpvsbc}
If $\mathrm{bc}(G)=m$, then $\mathrm{bp}(G)\leq \frac{1}{2}(3^m-1).$ 
\end{theorem}

\begin{proof}

Fix an $m$-cover of $G$, $\kappa=\{B_1, \dots, B_m\}$, and for each biclique $B_i$ we label the vertex classes  as class 0 and class 1. We now represent each vertex of $G$ as an element of $\{0,1,*\}^m$, based on its membership of elements of $\kappa$. If $v\in V(G)$, define $\tilde{v}\in \{0,1,*\}^m$ by:
\[\tilde{v} _i=\begin{cases} 0 & $if $v$ is in  class  0 of $B_i,\\
                      1 & $if $v$ is in class 1 of $B_i,\\
                      * & $otherwise$. \end{cases} \]

If $\tilde{u}=\tilde{v}$ then $u$ and $v$ have the same neighbours. Therefore, if $u \neq v$, this would mean that $\mathrm{bp}(G)=\mathrm{bp}(G-u)$, so we could replace $G$ with $G-u$ and thus we may assume vertices all have different representations. Using the identification $v \mapsto \tilde{v}$, $V(G)\subseteq\{0,1,*\}^m$.

Since the biclique partition number of a graph is at least that of each of its induced subgraphs, we may assume that $V(G)= \{0,1,*\}^m$. Note that there is an edge between two vertices $u$ and $v$ if and only if there is some $i$ for which $\{ u_i, v_i\}=\{0,1\}$. In other words, this is the complement of the intersection graph of all subcubes of an $m$-dimensional cube. We write $G_m$ for this graph.

Before proceeding by induction on $m$, we introduce some notation.  We write $(\alpha,\beta)$ for the biclique with vertex classes $\alpha$ and $\beta$. If $\alpha \subseteq \{0,1\}^m$ and $x,y \in \{0,1,*\}$, we define the following subsets of $\{0,1,*\}^{m+1}$:

\[ x\alpha:=\{(x,v_1, v_2, \dots, v_m): (v_1, \dots, v_m) \in \alpha\}
\quad \begin{smallmatrix} x \\ y \end{smallmatrix} \alpha :=x\alpha \cup y\alpha
\]

Let $\pi_m$ be an optimal partition of $G_m$---i.e. $\pi _m$ is a $\mathrm{bc}(G)$-cover of $G$.  Add $B_1$ to $\pi_{m+1}$. For every $(\alpha, \beta)\in \pi_m$, place the following three bicliques in $\pi_{m+1}$:
($\begin{smallmatrix} *\\ 0 \end{smallmatrix} \alpha, \begin{smallmatrix} *\\ 0 \end{smallmatrix} \beta$), 
($\begin{smallmatrix} *\\ 1 \end{smallmatrix} \alpha, 1 \beta$) and
($1  \alpha,  *  \beta$). 

\begin{claim}
$\pi_{m+1}$ is a partition of $G_{m+1}$ and contains $3 \mathrm{bp}(G_m)+1$ bicliques.
\end{claim}

\begin{proof}[Proof of claim]
Let $u=(u_1, \dots, u_{m+1})$ and $v=(v_1,\dots, v_{m+1})$ be two vertices that are not joined by an edge of $G_{m+1}$. This means there is no $i$ such that $\{u_i, v_i\}=\{0,1\}$ so $uv$ is not an edge of $B_1$. Moreover, if we define $u'=\{u_2, \dots, u_{m+1}\}$ and $v'=\{v_2, \dots v_{m+1}\}$, i.e. $u$ and $v$ with the first coordinate removed, then $u'v'$ is not an edge of $G_m$. Hence $u'v'$ is not an edge in any biclique of $\pi_m$; this implies $uv$ is not an edge of any biclique of $\pi_{m+1}$.

Conversely, let $e$ be an edge of $G_{m+1}$. If $e$ is also an edge of $B_1$, it is contained in no other bicliques of $\pi_{m+1}$. If $e$ is not an edge of $B_1$, then removing the first coordinate from each of its endpoints forms an edge $e'$ of $G_m$. As $\pi_m$ is a partition of $G_m$, $e'$ lies in exactly one $(\alpha, \beta) \in \pi_m$. Then, by inspection, $e$ must lie in exactly one of the corresponding bicliques in $\pi_{m+1}$.
\end{proof}
As $\mathrm{bp}(G_1)=1$, our proof is concluded by applying the above claim inductively.

\end{proof}

\begin{theorem}
There is a graph G with $\mathrm{bc}(G)=m$, but $\mathrm{bp}(G)=\frac{1}{2}(3^m-1)$.
\end{theorem}

\begin{proof}
For a graph $G$, we write $A(G)$ for the adjacency matrix of the graph, and $\mathrm{rank}(G)$ for the rank of $A(G)$. Suppose $P_1,\dots, P_k$ form a biclique partition of $G$. Then $ A(G)=\sum_{i=1}^k A(P_i)$. Since rank is subadditive and bicliques have rank 2, $\mathrm{rank}(G)\leq 2k$. This implies $\mathrm{bp}(G)\geq \mathrm{rank}(G)/2$.

We conclude the proof by calculating $\mathrm{rank}(G_m)$, where $G_m$ is the same graph as defined above.   We define the following order: $0<1<*$ and use lexicographical order for the product spaces $\{0,1,*\}^m$, for all $m$. 

Note that $A_m=A(G_m)$ has $3^m$ rows (and columns) and the final row (and column) in the above order consists solely of 0's. We shall show by induction on $m$ that the first $3^m-1$ rows of $A_m=A(G_m)$, together with the all 1 vector, form a linearly independent set. This is trivial for $m=1$.

In block matrix notation, and using the convention that $A_0=\begin{pmatrix} 0 \end{pmatrix},$
\[A_m= \begin{pmatrix} A_{m-1} & 1 & A_{m-1} \\ 1 & A_{m-1} & A_{m-1} \\ A_{m-1} & A_{m-1} & A_{m-1} \end{pmatrix} \rightarrow \begin{pmatrix} 0 & 1-A_{m-1} & 0 \\ 1-A_{m-1} & 0 & 0 \\ A_{m-1} & A_{m-1} & A_{m-1} \end{pmatrix}.
 \]

The second matrix  is obtained from the first by simple row operations. We define $A'_m$ as the matrix formed by replacing the final row of this second matrix by a row of only 1's.
\[A'_m= \begin{pmatrix} 0 & 1-A_{m-1} & 0 \\ 1-A_{m-1} & 0 & 0 \\ A'_{m-1} & A'_{m-1} & A'_{m-1} \end{pmatrix}
\]

Then our induction hypothesis is equivalent to the claim that $A'_m$ is of full rank---in other words the linear span of the rows of $A'_m$, $\mathrm{span}(A'_m)$, is $\mathbb{R}^{3^m}$. Since $\mathrm{span}(1-A_{m-1})=\mathrm{span}(A'_{m-1})=\mathbb{R}^{3^{m-1}}$, $\mathrm{span}(A'_m)=\mathbb{R}^m$. Therefore, $\mathrm{rank}(A_m)\geq 3^m-1$ and $\mathrm{bp}(G_m)\geq \frac{1}{2}(3^m-1)$. Combined with Theorem \ref{bpvsbc}, this concludes the proof.

\end{proof}

\section{Relationship between $\mathrm{lbc}$ and $\mathrm{lbp}$} \label{cantboundlbp}

As before, we write $(\alpha, \beta)$ for a biclique with vertex classes $\alpha$ and $\beta.$

In a contrast to the previous section's results, we show that $\mathrm{lbp}$ cannot be bounded by $\mathrm{lbc}$, indeed, even for $\mathrm{lbp}(G)=2$, it is possible for $\mathrm{lbc}(G)$ to be arbitrarily large. More concretely:

\begin{theorem} \label{main}
For all $m\geq 2$, there is a graph $G$ with a 2-local $m$-cover for which $\mathrm{lbp}(G)\geq \frac{1}{2} \log(\frac{m-1}{3})$.
\end{theorem}

We explicitly construct graphs satisfying this inequality.
For $m\geq 2$, let $G_m$ have vertex set 
\[V(G_m)=\big\{v \in \{0,1,*\}^m \big| \, v \text{ has exactly $m-2$ $*$'s} \big\}.\]

  We let $uv\in E(G_m)$ if and only if  there is some $i$ for which $\{ u_i, v_i\}=\{0,1\}$, where $u=\{u_1, \dots, u_m\}$ and $v=\{v_1, \dots, v_m\}$. Although we do not use this fact, $G_m$ is the complement of the intersection graph of all codimension 2 subcubes of a cube of dimension $m$. Our definition of $G_m$ may appear strange but it arises from simple reverse engineering of a graph with a 2-local $m$-cover. See the proof of Theorem \ref{upperbound} for more details. 
  
The graph $G_m$ may also be viewed as the union of $m$ bicliques, $B_1, \dots, B_m$, where $B_i$ is the subgraph induced by the vertex set $\{v \in V(G)|\, v_i \neq *\}$. These $B_i$ are known henceforth as \emph{covering bicliques}, as they form a 2-local cover, although not partition, of $G_m$. Thus $\mathrm{lbc}(G_m)=2$. 

A \emph{crown graph} on $2t$ vertices, $H_t$, is $K_{t,t}$ with a perfect matching removed. We may label the vertices of $H_t$ as $u^1,\dots, u^t$ and $v^1, \dots, v^t$;  $u^iv^j$ is an edge if and only if $i\neq j$, and there are no other edges.

We shall show that an $r$-local partition of $G_m$ can be altered to form an $r$-local partition of $H_t$, for $t$ linear in $m$. We then show that the local biclique cover number of crown graphs tends to infinity as the size of the graph does, indeed, $\mathrm{lbc}(H_t)=\Omega(\log t)$. Since $\mathrm{lbp} \geq \mathrm{lbc}$, this suffices to finish the proof.

\begin{proof}[Proof of Theorem \ref{main}]
Let $G_m$ and its subgraphs $B_1, \dots, B_m$ be as above. We say that an edge $uv$ is \emph{shared} if it is an edge in two of the covering bicliques; equivalently there are two $i$ such that $\{u_i, v_i\}=\{0,1\}$.

Let $\pi$ be a partition of $G$---we call bicliques in $\pi$ \emph{partitioning bicliques}. Let $e$ be a shared edge contained in $B_i$ and let $A=(\alpha,\beta)$ be the partitioning biclique containing it. Consider the graph $A\cap B_i$---i.e. the graph with vertex set $V(A)\cap V(B_i)$, and edge set $E(A)\cap E(B_i)$.   It consists of up to two disjoint bicliques---let $A^e_i=(\alpha^e_i, \beta^e_i)$ denote the one containing the edge $e$. We term this the \emph{restriction} of $A$ onto $B_i$ (with respect to the edge $e$). Note that all vertices in $\alpha^e_i$ agree in the $i^{th}$ coordinate, as do all vertices in $\beta^e_i$.

%DOUBLE CHECK THIS PROOF AS JUST DID A CHANGE IN TERMINOLOGY.
 %PERHAPS ITS UNFORTUNATE THAT WE TALK ABOUT STAR GRAPHS, BUT ALSO ONE OF THE SYMBOLS USED IS A $*$

\begin{lemma}
Let $e$ be a shared edge of $B_i$ and $B_j$, and $A=(\alpha, \beta)$ a biclique containing it. Then one of $A^e_i$ or $A^e_j$ is a star, or both $A^e_i$ and $A^e_j$ are $K_{2,2}$.

\end{lemma}

This lemma, though slightly technical, is key to our approach. A more intuitive formulation is that all bicliques containing an edge shared by two covering bicliques are either  `close' to being contained in one of these covering bicliques or `very small'. 

\begin{proof}
By permuting the coordinates and swapping 0 with 1 in the first two coordinates, we may assume that, the endpoints of $e$ are $(0,0,*, \dots ,*)$ and $(1,1,*, \dots,*)$, with the former being in $\alpha$ and the latter in $\beta$.

Notice that if any of $\alpha^e_1,$  $\alpha^e_2,$ $\beta^e_1$ or  $\beta^e_2$ are singletons, we are done, so we assume the contrary. Suppose at least one of $(0,1,*, \dots ,*)$ and $(1,0,*,\dots, *) $ is in $A$---without loss of generality $(0,1,*, \dots, *)$ is in $\alpha$. Then every vertex in $\beta$ must have a 1 in the first coordinate. But any vertex in $\beta_2^e$ has 1 in the second coordinate. Thus there is only one vertex in $\beta_2^e$, contradicting our earlier assumption.

Therefore neither  $(0,1,*, \dots, *)$ nor $(1,0,*, \dots, *)$ is in $A$. A second vertex in $\alpha^e_1$ must have a 0 in the first coordinate, a $*$ in the second coordinate and without loss of generality, a 0 in the third coordinate. Equally, a second vertex in $\beta_2^e$ must be $(*,1,1,*,\dots, *)$ in order to be adjacent to the vertices we have assumed are in $\alpha$. 

A further vertex in $\alpha_2^e$  must start with $*$, followed by 0 and have a further non-asterisk digit, in the $i^{th}$ place, say, ($i>2$). Our previous assumptions determine that there is only one further vertex in $A$---the vector with 1's in the first and $i^{th}$ coordinates. It can be seen that $A_1$ and $A_2$ are both $K_{2,2}$.

\end{proof}

We now define for each $i$, a colouring $c_i$ of the shared edges of $B_i$. If $e$ is a shared edge of $B_i$ and $B_j$ and $A$ is the partitioning biclique containing it, $c_i(e)$ is red if $A^e_i$ is a star or a $K_{2,2}$ and blue otherwise. Note that the preceding lemma shows that if $c_i(e)$ is blue then $A_j^e$ is a star.

A shared edge is blue in at most one colouring. Hence, each of the $2\binom{m}{2}=m(m-1)$ shared edges is red in at least one of the $m$ colourings, implying that at least one of the $m$ covering bicliques, $B_1$ say, must contain at least $m-1$ shared edges coloured red.

All shared edges are vertex-disjoint---as seen by the form of the edges as vectors---so we may label the edges as $u^1 v^1, \dots, u^{m-1} v^{m-1}$, where the $u^i$ and the $v^j$ are from different classes of $B_1$. We seek a large $B\subseteq B_1$ such that all shared edges in $B$ are in partitioning bicliques that restrict to stars on $B_1$. Pick $I \subseteq [m-1]$ as follows. Let $1\in I$. If the shared edge $e=u^1 v^1 \in A \in \pi$ and $A^e_1 \cong K_{2,2}$ then discard from $I$ the (at most two)  indices other than 1 for which one of $u^i$ and $v^i$ is a vertex of this $K_{2,2}$. Then proceed to the next index not already discarded, and continue in the same way. Since at each stage, we throw away at most two indices, the set of surviving indices, $I$, satisfies $|I| \geq (m-1)/3.$
%IS THIS VAGUE? WRITE `LIKE A PROGRAM'?

Let $B$ be the induced subgraph of $B_1$ with the $u^i$ and $v^j$ as vertices, for $i \in I,$ $j \in I$.  Let $H$ be the subgraph of $B$ formed by removing all shared edges, $u^i v^i$ for $i\in I$. Clearly, $H$ is isomorphic to $H_{|I|}$, the crown graph on $2|I|$ vertices. Let $A$ be a biclique containing one of these removed edges, $e$, and note that $A\cap B$ is the union of at most two disjoint bicliques. By the construction of $B$, the component of $A \cap B$ containing $e$, $A_1^e\cap B$, is a star. We label the other (possibly empty) biclique component $A'$. From $\pi$, we induce a partition $\pi'$ of $H$ as follows. If $A \in \pi$ contains an edge of $H$, then place $A\cap H$ inside  $\pi'$. Now, $A \cap H=A\cap B$ and is hence a union of bicliques, unless $A$ contains a removed edge. Using $G-e$ to denote the graph $G$ with the edge $e$ deleted, in this case, $A \cap H=A'\cup (A^e_1 \cap B-e)$, the latter component remaining a star. Note that if $\pi$ is a $k$-local partition of $G_m$, then $\pi'$ is a $k$-local partition of $H$. This argument is valid for all partitions $\pi$ of $G$, so $\mathrm{lbp}(H)\leq \mathrm{lbp}(G_m)$.

We now complete the proof in the following lemma:

\begin{lemma}
$\mathrm{lbc}(H_t)\geq \frac{1}{2} \lceil \log t \rceil$.
\end{lemma}

\begin{proof}

A biclique cover of $H_t$ induces a biclique cover of $K_t$ when two corresponding vertices (i.e. $u^i$ and $v^i$ for each $i$) are identified as one. As mentioned earlier, $\mathrm{lbc}(K_t)=\lceil \log t \rceil$ so, there must be some vertex of the $K_t$ in at least $\lceil \log t \rceil$ bicliques belonging to the induced cover. This implies one of the vertices of $H_t$ that are identified to it is in at least  $\frac{1}{2} \lceil \log t \rceil$ bicliques, in the original covering.

\end{proof}

\end{proof}

%\begin{observation} \label{obs} In fact, we have shown that $\mathrm{lbp}(G_m)\geq \frac{1}{2} \log(\frac{m-1}{3}).$ Thus for all $m$, there is $G$ with a 2-local $m$-cover for which $\mathrm{lbp}(G)\geq \frac{1}{2} \log(\frac{m-1}{3})$.
%\end{observation}

\section{Bounding $\mathrm{lbp}(G)$, for $\mathrm{lbc}(G)=2$ } \label{boundinglbp}

The previous section showed that even for $\mathrm{lbc}(G)=2$, we cannot place an absolute bound on $\mathrm{lbp}(G)$. In this section, we  instead aim for bounds involving additional covering information about $G$. One way of doing this is by looking at the number of bicliques in a 2-local cover of $G$:

\begin{theorem} \label{upperbound}

If $G$ has a 2-local $m$-cover then $\mathrm{lbp}(G)\leq 2 \lceil \log (m-1) \rceil+2$.

\end{theorem}

This upper bound is best possible up to a constant factor by Theorem \ref{main}. We proceed by using an argument similar to the start of the proof of Theorem \ref{bpvsbc} to reduce proving the upper bound for all graphs with a 2-local $m$-cover to proving it for one particular graph. This `worst' graph will turn out to be similar to the graph used in the proof of Theorem \ref{main}.

\begin{proof}

Fix a 2-local $m$-cover of $G$. We may identify vertices $v \in V(G)$ with vectors $\tilde{v} \in \{0,1,*\}^m$ with exactly $m-2$ $*'$s, similar to the proof of Theorem \ref{bpvsbc}. As shown there, we may assume no two vertices have the same representation, and that all such vectors are the representation of some vertex, i.e. we assume that          
\[V(G)\subseteq\{v \in \{0,1,*\}^m|\,  v \text{ has at least $m-2$ $*$'s} \}. \]
Note that there is an edge between two vertices $u$ and $v$ if and only if there is some $i$ for which $\{ {u_i}, {v}_i\}=\{0,1\}$.

For $i=1, \dots, m$, let $B_i$ be the induced subgraph on the vertices whose $i^{th}$ coordinate is not $*$. We define shared edges in an identical manner to in Section 4 and we shall define $C_i$, a subgraph of $B_i$, such that every edge of $G$ is an edge in exactly one $C_i$.

Let $C_i$ contain all non-shared edges of $B_i$. Additionally, for all $i$ and $j$, let $C_i$ contain exactly one of the  edges shared between $B_i$ and $B_j$, and let $C_j$ contain the other. Thus any vertex is in at most two $C_i$, and the $C_i$ are isomorphic, so $\mathrm{lbp}(G)\leq 2 \mathrm{lbp}(C_1)$. But $C_i$ is a complete bipartite graph minus a matching of size ${m-1}$. It is easy to show that $\mathrm{lbp}(C_1) \leq  \mathrm{lbp}(H_{m-1})+1,$ where $H_{m-1}$ is the crown graph on $2(m-1)$ vertices.

We now show that $\mathrm{lbp}(H_t)\leq \lceil \log t \rceil$, for any $t$. We name the classes of $H_t$ $U$ and $V$. We assign to  each vertex a unique label consisting of its class name and a binary vector of length $\lceil\log t \rceil$. For each position, $i=1, \dots, \lceil\log t \rceil$, we generate two bicliques corresponding to the difference in the $i^{th}$ position of the binary vector. More specifically, we let $B_i^1$ be the induced biclique formed by all the vertices of class $U$ with a 0 in the $i^{th}$ coordinate and all vertices of class $V$ with a 1 in the $i^{th}$ coordinate. Similarly, we let $B_i^2$ be the induced biclique formed by all vertices of class $U$ with a 1 in the $i^{th}$ coordinate and all vertices of class $V$ with a 0 in the $i^{th}$ coordinate. The collection of all these bicliques is a $\lceil\log t \rceil$-local partition of $H_t$. This concludes the proof.

%As usual, we label the vertices of $H_{t}$ as $u^1, \dots, u^{t}$ and $v^1, \dots, v^{t}$ with $u^i$ adjacent to $v^j$ if and only if $i\neq j$. Label the vertices of $K_t$ as $1,\dots , t$. Let $\pi$ be a  $\lceil \log t \rceil$ -local partition of $K_t$.  We form a partition $\pi'$ of $H_t$ as follows: if $(I,J)$ is a biclique in $\pi$, let $(U_I, V_J),$ $(U_J,V_I)$ both be in $\pi'$, where $U_I=\{u^i: i \in I \}$, and similarly for $U_I, V_J$ etc. It is easy to see that $\pi'$ is a $\lceil \log t\rceil$-local partition of $H_t$.

\end{proof}

Since $\mathrm{bc}(G)$ is a measure of independent interest to local measures, one may try to use the separate existence of a $m$-cover in a graph with $\mathrm{lbc}(G)=2$, rather than the existence of a 2-local $m$-cover, to bound $\mathrm{lbp}(G)$. More formally, we ask:

%DOES THE FIRST SENTENCE MAKE SENSE?

\begin{question}
What is the smallest $k(m)$ such that if the graph $G$ has $\mathrm{bc}(G)=m$ and $\mathrm{lbc}(G)=2$, we have $\mathrm{lbp}(G)\leq k$?  
\end{question}

Theorem \ref{main} tells us that $k(m) \geq \frac{1+o(1)}{2} \log m$. However, an upper bound does not follow directly from Theorem \ref{upperbound} as the cover that shows $\mathrm{lbc}(G)=2$ may be genuinely different from the cover that shows $\mathrm{bc}(G)=m$. The following result shows this:

\begin{theorem}

For all $m \geq 4$, there is a bipartite graph, $G$, with $\mathrm{lbp}(G)=\mathrm{lbc}(G)=2$ and $\mathrm{bp}(G)=\mathrm{bc}(G)=m$, but there are no $(m-1)$-local $m$-covers. 

\end{theorem}

\begin{proof}

First, we make a general observation about biclique covers. Two edges, $ \{v_1, v_2\}$ and $\{u_1, u_2\} \in E(G)$ are called \emph{strongly independent} if they are independent and the minimum degree of $G[u_1, u_2, v_1, v_2]$ is 1---in other words, they are vertex disjoint and do not lie in the same $K_{2,2}$ subgraph. No biclique can contain both edges so if $E(G)$ contains a set of $k$ pairwise strongly independent edges, $\mathrm{bc}(G) \geq k$.

%either $u_1$ or $u_2$ are adjacent to neither  $v_1$ nor $v_2$ in $E(G)$ or the other way round.

Let $X=\{x_1, \dots, x_m\},$  $Y=\{y_{i,j}: 1\leq i< j\leq m\}\cup \{y_{[m]}\}$  and let $V(G)=X\cup Y$. Let there be an edge from $x_i$ to $y_{i,j}$ and to $y_{j,i}$ for all appropriate $j$. Also, let $y_{[m]}$ have edges to all of the $x_i$. This defines a bipartite graph, $G$.  Let the partition $\pi$ consist of a $K_{1,m}$ with its centre at $y_{[m]}$, and stars at each of the $x_i$ each with leaves at all adjacent vertices other than $y_{[m]}$. Thus $\mathrm{lbp}(G)=\mathrm{lbc}(G)=2$. Note that $I=\{ \{x_1, y_{1,2}\}, \{x_2, y_{2,3}\}, \dots, \{x_m, y_{m,1}\}\}$ is a set of strongly independent edges. Therefore $\mathrm{bc}(G)\geq m$, and so $\mathrm{bp}(G)=\mathrm{bc}(G)=m$ as we can cover (indeed partition) with stars that have their centre in $X$. 

Suppose $\kappa$ is a biclique cover of $G$ using just $m$ bicliques. Each of these contains exactly one of the strongly independent edges listed above.

We can see that $I\cup \{x_1, y_{1,3}\} \setminus \{x_1, y_{1,2}\}$ is also a strongly independent set of $m$ edges. Therefore, $x_1$ must be at the centre of a star in $\kappa$. Similarly, every $x_i$ must be the centre of a star in $\kappa$, and these $m$ stars account for all the bicliques in $\kappa$. So the star with centre $x_i$, must contain all vertices adjacent to $x_i$. Thus $y_{[m]}$ is in $m$ bicliques of $\kappa$.

\end{proof}

So we can see that insisting that a cover of a graph $G$ attains $\mathrm{bc}(G)$ may ensure the cover is not close to being optimal for $\mathrm{lbc}$. %(IS THIS SENTENCE CLEAR?)
This leads to the following question---if we insist that the cover of a graph attains $\mathrm{lbc}(G)$, could this ensure that it is far from being optimal for $\mathrm{bc}$? More formally:

\begin{question} Suppose $\mathrm{bc}(G)=m$, $\mathrm{lbc}(G)=k$, what is the smallest $r=r(m,k)$ such that we can guarantee $G$ has a $k$-local $r$-cover?

\end{question}

An answer to this for $k=2$, combined with Theorem \ref{upperbound}, would lead to a partial solution of Question 1.

The following example of the hypercube is instructive, in giving bounds on the problem---in particular, it shows $r(2^{d-1}, d/2) \geq d2^{d-3}$, for $d$ even.

\begin{example}
Consider $Q_d$, for $d$ even. It is easy to show that $\mathrm{bp}(Q_d)=\mathrm{bc}(Q_d)= 2^{d-1}$. Indeed, each biclique contains at most $d$ edges, and $e(Q_d)=d2^{d-1}$, so $\mathrm{bc}(Q_d)\geq 2^{d-1}$. On the other hand, we may partition using $|Q_d|/2=2^{d-1}$ stars since $Q_d$ is bipartite. Dong and Liu \cite{dongliu} showed that $\mathrm{lbp}(Q_d)=\mathrm{lbc}(Q_d)=d/2$. In fact, their proof showed that a cover achieves this only if it consists solely of ${K_{2,2}}$'s so every $d/2$-local cover contains at least $d2^{d-3}$ bicliques.

\end{example}

\subsection*{Acknowledgements}

I would like to thank Robert Johnson for useful advice and many helpful discussions. I would also like to thank the anonymous referee for many helpful comments that helped improve the presentation of this paper. This research was funded by an EPSRC doctoral studentship.

%\section{bounding bp in terms of bc}

%\begin{theorem}
%For all graphs, $G$, $bp(G) \leq 3^bc(G)$, and this is tight.

%\end{theorem}
%Suppose $G_m$ is a graph with bc$(G_m)=m$. Fix an $m$-cover of $G_m$, $B_1, \dots, B_m$. Represent each vertex, v, in G, by a vector $v'\in \{0,1,*\}^m$, where $v'_i=\cases 1 if v in first class of B_i, 0 if v in second class of B_i and * otherwise$. Now, we may assume that each label occurs only once (and exactly once as we are interested in the worst case).
%Inductively we show that the rank of $G_m$ is $3^m-1$. 
%Then we exhibit an inductive partition of G.

\end{document}